\documentclass{amsart}
\usepackage[utf8]{inputenc}
\usepackage{amsmath, amsfonts, amssymb,amscd, amstext, amsthm, 
mathtools,enumerate}
\usepackage{hyperref}
\usepackage[pdftex]{color}

\theoremstyle{plain}

\newtheorem{lemma}{Lemma}[section]
\newtheorem{proposition}{Proposition}[section]
\newtheorem{corollary}{Corollary}[section]

\newtheorem{example}{Example}
\newtheorem*{mydef}{Teorema A}
\newtheorem*{mydef2}{Theorem B}
\newtheorem*{mydef3}{Theorem C}
\newtheorem*{mydef4}{Theorem D}
\newtheorem*{mydef5}{Theorem E}
\newtheorem*{mydef6}{Theorem F}

\theoremstyle{definition}
\newtheorem{definition}{Definition}[section]

\theoremstyle{remark}
\newtheorem*{remark}{Remark}

\title{Counting and Hausdorff measures for integers and $p$-adic integers}

\author{Davi Lima}

\address{Davi Lima: Instituto de Matem\'atica, Universidade Federal de Alagoas - UFAL, Av. Lourival Melo Mota s/n, Maceio, Alagoas, Brazil, 57072-970}
\email{davi.santos@im.ufal.br}
\author{Alex Zamudio Espinosa}

\address{Alex Zamudio Espinosa: Universidade Federal Fluminense - UFF, Rua Miguel de Farias 9,
Niter\'oi - Rio de Janeiro, Brazil
24220-900}
\email{ alexmze@id.uff.br.}

\begin{document}

\begin{abstract}
    In this work, we aim to advance the development of a fractal theory for sets of integers. The core idea is to utilize the fractal structure of $p$-adic integers, where $p$ is a prime number, and compare this with conventional densities and counting measures for integers.. Our approach yields some results in combinatorial number theory. The results show how the local fractal structure of a set in $\mathbb{Z}_p$ can provide bounds for the counting measure for its projection onto $\mathbb{Z}$. Additionally, we establish a relationship between the counting dimension of a set of integers and its box-counting dimension within $\mathbb{Z}_p$. Since our results pertain to sets that are projections of closed sets in $\mathbb{Z}_p$, we also provide both necessary and sufficient combinatorial conditions for a set $E \subset \mathbb{Z}$ to be the projection of a closed set in $\mathbb{Z}_p$.
\end{abstract}

\maketitle

\section{Introduction}

There are several notions of how to measure subsets of $\mathbb{Z}$, for instance, the asymptotic and upper Banach density. Let us remember the definition of upper Banach density of a set $E\subset \mathbb{Z}$. We denote $I=(M,N]=\{M+1,M+2,...,N\}\subset \mathbb{Z}$ an interval of integers and $|I|=N-M$ its length. The upper Banach density of $E$ (sometimes just Banach density of $E$) is defined as
            \begin{equation}
            d^*(E)=\limsup_{|I|\to +\infty}\dfrac{|E\cap I|}{|I|},
            \end{equation}
        where we write $|A|$ for the number of elements of a set $A.$
The Banach density can be used as a tool to compare subsets of $\mathbb{Z}$, a ``bigger" subset would have bigger Banach density. But when we deal with sets having zero Banach density then we need other tools to distinguish them. This is similar to when we have two sets of real numbers with zero Lebesgue measure, in such case one usual tool to compare them is the Hausdorff dimension and Hausdorff measures. Following this idea, Y. Lima and C. G. Moreira, contributed with a Theory of fractal subsets in $\mathbb{Z}$, in \cite{YG}, introducing the notion of $s$-counting measure and Counting Dimension of a set $E\subset \mathbb{Z}$. Notions of dimension for subsets of $\mathbb{Z}$ and $\mathbb{Z}^d$ also had been studied before Lima and Moreira by J. Naudts \cite{N1, N2},
T. Barlow and S. Taylor \cite{BT}, A. Iosevich, M. Rudnev, and I. Uriarte-Tuero \cite{IRU}. More recently, D. Glasscock in \cite{G} extended the work of Y. Lima and C. G. Moreira to the context of $\mathbb{Z}^d$. In \cite{GMR}, D. Glasscock, J. Moreira and F. Richter study aspects of geometric transversality in sets of $\mathbb{Z}$ using a machinery of fractal geometry of $\mathbb{R}^d$. In the first moment the results, following the results in \cite{YG} we are induced to imagine the counting dimension as an analogue of the Hausdorff dimension, this is the spirit of the Marstrand Theorem for subsets of integers proved por Lima-Moreira. However, our work explores a different perspective on the counting dimension, suggesting that it is actually more closely related to the box-counting dimension. This interpretation is outlined in Theorem E below.

The aim of this paper is to advance this line of thought by introducing a novel method for "measuring" and estimating the dimension of subsets of $\mathbb{Z}$. Our approach is straightforward: consider a prime number $p$ and $\mathbb{Z}_p$, the set of $p$-adic integers. We view a set $E \subset \mathbb{Z}$ as a subset of $\mathbb{Z}_p$, take its $p$-adic closure, and investigate the connections between its fractal-geometric properties such as Hausdorff measures, box and Hausdorff dimensions and its analytical-combinatorial properties, such as densities, counting measures and counting dimension


In this paper we denote by $\mathcal{P}(X)$ the power set of $X$. We also write the set of natural numbers by $\mathbb{N}=\{1,2,3,...\},$ the non-negative integers $\mathbb{N}_0=\{0,1,2,...\}$, the set of integers numbers $\mathbb{Z}=\{...,-2,-1,0,1,2,...\},$ and non-positive integers $\mathbb{Z}^{-}=-\mathbb{N}_0=\{...,-2,-1,0\}$. Consider $\mathcal{F}_p: \mathcal{P}(\mathbb{Z})\rightarrow \mathcal{P}(\mathbb{Z}_p)$ given by
        \begin{equation}
           \mathcal{F}_p(E)=\overline{E}^p
        \end{equation}
        this set is referred to as the closed embedding of $E$ in $\mathbb{Z}_p$. 
In the same way that we have the embedding $\mathcal{F}_p$ we also consider the projection $\Pi_p: \mathcal{P}(\mathbb{Z}_p)\rightarrow \mathcal{P}(\mathbb{Z})$ given by   
            \begin{equation}
            \Pi_p(F)=F\cap \mathbb{Z}.
            \end{equation}
            
        Then, our idea is to consider the set $\Pi_p(\mathcal{F}_p(E))$ and  relate fractal-geometric properties of it and analytical-combinatorial properties of $E$. We have $E\subset \Pi_p(\mathcal{F}_p(E))$, on the other hand we cannot expect that equality holds because new integers elements may arise in $\mathcal{F}_p(E)$. We give some examples in subsection \ref{subsection nmap}. This same observation, when we want to study projections of sets entirely contained in $\mathbb{N}$, forces us to consider the $p$-adic projection $\Pi_p$ slightly different: we consider $\Pi_p^{\ast}(F)=F\cap \mathbb{N}$. If fact, if $E\subset \mathbb{N}$ then it is more reasonable to take 
        $$\Pi^{\ast}_p(\mathcal{F}_p(E))=\mathcal{F}_p(E)\cap \mathbb{N}$$ 
        instead of $\mathcal{F}_p(E)\cap \mathbb{Z}$,
        however we continue using $\Pi_p$ by simplicity. It follows from the above considerations that the classes
            $$\mathcal{S}_p=\{E\subset \mathbb{Z}; E=\Pi_p(\mathcal{F}_p(E))=\mathcal{F}_p(E)\cap \mathbb{Z}\}$$
            or
            $$\mathcal{S}^{\ast}_p=\{E\subset \mathbb{N};E=\Pi_p(\mathcal{F}_p(E))=\mathcal{F}_p(E)\cap \mathbb{N}\}$$
            are important in the development of this work.
         In any case we have that $\mathcal{S}_p$, $\mathcal{P}(\mathbb{Z})\setminus \mathcal{S}_p$, $\mathcal{S}^{\ast}_p$, $\mathcal{P}(\mathbb{N})\setminus \mathcal{S}^{\ast}_p$ are uncountable classes\footnote{To see why $\mathcal{P}(\mathbb{Z})\setminus \mathcal{S}_p$, $\mathcal{P}(\mathbb{N})\setminus \mathcal{S}^{\ast}_p$ are uncountable, notice that given any infinite subset $W\subset \mathbb{N}$ one has that $\{1+p^{n}:\,\, n\in W\}$ is not in $\mathcal{S}_p$ or $\mathcal{S}^{\ast}_p$.}. It is important to provide sufficient and necessary combinatorial conditions under which a subset $E\subset \mathbb{N}$ is in $\mathcal{S}^{\ast}_p$ or $E\subset \mathbb{Z}$  belongs to $\mathcal{S}_p.$ We address some sufficient and necessary combinatorial conditions in subsection \ref{subsection nmap}. 
        
        Since $E\subset \Pi_p(\mathcal{F}_p(E))$ we have $d^{\ast}(E)\le d^{\ast}(\Pi_p(\mathcal{F}_p(E)))$ for any $p$. The space $\mathbb{Z}_p$, with the usual metric, has Hausdorff dimension $1$ and the $1$-dimensional Hausdorff measure coincides with its Haar measure. We denote this Haar measure by $\mu_p$ and remind that on balls one has $\mu_p(B_{p^{-k}}(x))=p^{-k}$, that is, $\mu_p$ is an uniformly distributed measure, in section \ref{sec2} we describe these objects. 
Our first result, Theorem A, which is indeed similar to an instance of Furstenberg's Correspondence Principle, shows that the upper Banach density and the Haar measure can be related. The boundary of a set $F\subset \mathbb{Z}_p$ is denoted by $\partial F$.
        
\begin{mydef}\label{P1}
        Consider $p$ a prime number. If $F\subset \mathbb{Z}_p$ is a closed set then  $$d^*(\Pi_p(F))\le \mu_p(F).$$
        If $\mu_p(\partial F)=0$ then the equality holds.

        \end{mydef}
The arithmetic difference between two sets in $\mathbb{Z}$, say, $A$ and $B$ is
    $$A-B=\{m-n;m\in A, n\in B\}.$$
The following corollary can be obtained by combinatorial simple methods but we write as a consequence of our first result.
        \begin{corollary}\label{cor.1.1}
 Consider $E, F\subset \mathbb{Z}$ such that $d^{\ast}(E)+d^{\ast}(F)>1$ then for every prime number $p$ and for every $n
 \in \mathbb{N}$ the arithmetic difference $E-F$ contains multiples of $p^n$. 
\end{corollary}

The notion of upper Banach density can be generalized by the $s$-counting measure, see \cite{YG}. Given $E\subset \mathbb{Z}$, its $s$-counting measure is defined by
    $$\mathcal{H}_s(E)=\limsup_{|I|\to \infty}\dfrac{|E\cap I|}{|I|^s}$$
where $I$ runs through all intervals in $\mathbb{Z}$. These are global notions of how the set $E$ is distributed. We observe that $\mathcal{H}_1=d^{\ast}$. It is natural to ask if Theorem A is in fact more general if we try to change $d^*$ by $\mathcal{H}_s$, for a certain $s$, and $\mu_p$ by the $s$-Hausdorff measure in $\mathbb{Z}_p$. Pursuing such goal we obtained our second main result.

We can try to recover $\mathcal{H}_s(E)$ from the local behaviour of the closed embedding $\mathcal{F}_p(E)$ in some $p$-adic integer ring. 
First of all, given $F\subset \mathbb{Z}_p$ and  $n\in \mathbb{N},$ define
    $$\eta_s(F;x,n)=\limsup_{|I|\to \infty}\dfrac{|F\cap B_{p^{-n}}(x)\cap I|}{|B_{p^{-n}}(x)\cap I|^s}$$
    and
    $$\eta_s(F;x)=\limsup_{n\to \infty} \eta_s(F;x,n),$$
   we shall see in subsection \ref{subs3.4} that the $\eta_s$ is a local density with respect to $\mathcal{H}_s$. Denote by $\mathcal{H}^s$ the $s$-dimensional Hausdorff measure on Borel $\sigma$-algebra of $\mathbb{Z}_p$. In fact, Theorem A is generalized by



\begin{mydef2}\label{mydef2.0}
Let $F\subset \mathbb{Z}_p$ be a closed set and assume $\mathcal{H}^s(F)<\infty$. Suppose that there is a constant $C>0$ such that
\[\limsup_{|I|\to +\infty}\frac{|F\cap B\cap I|}{|B\cap I|^s} \le C,\]
for any ball $B\subset \mathbb{Z}_p$. Then we have
\[\mathcal{H}_s(\Pi_p(F))\le \int_{F} \eta_s(F;x) d\mathcal{H}^s(x).\]
\end{mydef2}
The uniform upper bound stated in Theorem B can be relaxed as demonstrated in the following theorem. However, an additional hypothesis is required to ensure the existence of a positive and finite constant that provides useful arithmetic information.
\begin{mydef3}\label{mydef2}
Let $F\subset \mathbb{Z}_p$ be a closed set such that $\mathcal{H}^s(F)< \infty$ and assume that there is an $\mathcal{H}^s$-integrable function $g$ such that $\eta_s(F;x,n)\le g(x)$ for $\mathcal{H}^s$-almost all $x$ in $F$. Then we can find an explicit constant $\overline{\theta}_{\ast}^s(F)\in [0,+\infty]$ such that 
\begin{equation}
\int_F \eta_s(F;x)d\mathcal{H}^s(x)\ge \overline{\theta}_{\ast}^s(F) \mathcal{H}_s(\Pi_p(F)).
\end{equation}
More is true: If $F$ is a $s$-Ahlfors-David regular set then $$0<\overline{\theta}_{\ast}^s(F)<+\infty.$$

\end{mydef3}

We can use Theorems B and C to give some applications in number theory. Fix $\mathcal{F}\subset \mathcal{P}(\mathbb{Z})$ and $s>0$. We shall say $E\subset \mathbb{Z}$ is $s$-abundant in $\mathcal{F}$ if there exists a sequence of finite sets $\{A_j\}_{j\in \mathbb{N}}\subset \mathcal{F}$ and $c>0$ such that $|A_{k+1}|>|A_k|$ and
       $$|E\cap A_k|\ge c\cdot |A_k|^s.$$
Note that if $E$ contains arithmetic progressions of arbitrarily length then $E$ is $1$-abundant on $\mathcal{A}_f$, the set of finite arithmetic progressions.  Consider $\mathcal{A}^r_f$ the set of arithmetic progressions with common difference $r$. Consider $p$ a prime number. We have the following result

\begin{mydef4}\label{cor1.3}
    Let $F\subset \mathbb{Z}_p$ be a closed set such that $\mathcal{H}_s(\Pi_p(F))>0$. Suppose that $\mathcal{H}^s(F)<+\infty$. If either  
    \begin{enumerate}
    \item[(a)] there is a constant $C>0$ such that
\[\limsup_{|I|\to +\infty}\frac{|F\cap B\cap I|}{|B\cap I|^s} \le C,\]
for any ball $B\subset \mathbb{Z}_p$ or
    \item[(b)] There is an  $\mathcal{H}^s$-integrable function $g$ such that $\eta_s(F,x, n)\le g(x)$ for $\mathcal{H}^s$-almost all $x\in F$ and $F$ is $s$-Ahlfors-David regular.
    \end{enumerate} 
    then, for an unbounded sequence $\{n_j\}_j$ of natural numbers we have that $\Pi_p(F)\subset \mathbb{Z}$ is $s$-abundant in $\mathcal{A}^{p^{n_j}}_f$. 
\end{mydef4}

In the direction of a fractal theory on sets of integers numbers, we also want to understand the idea of dimension. Following \cite{YG}, we define the counting dimension of a set $E\subset \mathbb{Z}$ as 
    $$D(E)=\limsup_{|I|\to \infty}\dfrac{\log |E\cap I|}{\log|I|}.$$
Write $s_0=D(E)$. By the definition of $\mathcal{H}_s(E)$, we have that $s_0\ge 0$ is such that $\mathcal{H}_s(E)=\infty$ for $s<s_0$, and $\mathcal{H}_s(E)=0$ for $s>s_0$. Sets with positive density have counting dimension 1. It is natural to ask if the counting dimension of a set $E$ should really be considered as analogous to the Hausdorff dimension of the set $\mathcal{F}_p(E)$. In the next theorem we will see that this is not always the case, we show that $D(E)$ can always be compared with the box-counting dimension of $\mathcal{F}_p(E)$. We also give an example in which the Hausdorff dimension of $\mathcal{F}_p(E)$ is strictly smaller than the counting dimension of $E$, contrary to what one would expect from theorems like Theorem B and C. We observe that the map $\mathcal{F}_p$ can be defined on $\mathcal{P}(\mathbb{Z}_p)$ in the same way: $F\mapsto \bar{F}^p$.


\begin{mydef5}
If $F\subset \mathbb{Z}_p$ then
    $$D(\Pi_p(\mathcal{F}_p(F)))\le BD(F),$$
   where $BD(F)$ is the (upper)Box-counting dimension of $F$. In particular, $D(E)\le BD(E)$ for any $E\subset \mathbb{Z}$. Moreover, there is a closed set $F\subset \mathbb{Z}_5$ such that 
    $$D(\Pi_5(F))=BD(F)>\underline{BD}(F)\ge HD(F),$$
     where $\underline{BD}(F)$ is the lower Box-dimension of $F$.
\end{mydef5}
In the proof of this theorem it will be clear that the prime $5$ is not special, we could have chosen another prime. 

Using the notion of Box-counting dimension and our approach in the subsection \ref{subsection nmap} have the following result

\begin{mydef6}
Let $E\in \mathcal{S}_p$ or $E\in \mathcal{S}^{\ast}_p$ such that $BD(E)=s$. Then for any $\delta>0$, for $n\in \mathbb{N}$ large enough there are at most $k_n\le p^{n(s+\delta)}$ arithmetic progressions, $A^{(n)}_1,...,A^{(n)}_{k_n}$ of common difference $p^n$ such that $E\subset \cup_{i=1}^{k_n}A^{(n)}_i$, $\cup_{i=1}^{k_{n+1}}A^{(n+1)}_i\subset \cup_{i=1}^{k_n}A^{(n)}_i$ and
    $$E=\bigcap_{n\ge 1}\bigcup_{i=1}^{k_n}A^{(n)}_i.$$
\end{mydef6}
Throughout this work if we consider $E\subset \mathbb{N}$ then the arithmetic progressions are of natural numbers. The above result tell us that any set $E\in \mathcal{S}_p$ is a countable intersection of open sets in the profinite topology of $\mathbb{Z}$.



\section{Preliminaries}\label{sec2}

In this section we present some preliminaries which are used in this work. Throughout the paper $p$ is a prime number. 

\subsection{The $p$-adic Integers}
Given $n\in \mathbb{Z}\setminus \{0\}$ we take $v_p(n)$ the $p$-adic valuation of $n$, that is, $n=p^{v_p(n)}\cdot r$ where $\mbox{gcd}(p,r)=1$ and $r\in \mathbb{Z}$, if $n=0$ one defines $v_p(0)=\infty$. If $x=\dfrac{a}{b}\in \mathbb{Q}$ then $v_p(x)=v_p(a)-v_p(b).$ Using $v_p$ we can define an absolute value on $\mathbb{Q}$ given by $|x|_p=p^{-v_p(x)}$, and from this one defines the ultrametric $d_p(x,y)=|x-y|_p$.
The completion of $(\mathbb{Q},d_p)$ is denoted by $\mathbb{Q}_p$ and this is the set of $p$-adic numbers. 

The unit ball
	$$\mathbb{Z}_p=\{x\in \mathbb{Q}_p;|x|_p\le 1\}$$
is called the set of $p$-adic integers and it has several remarkable properties.	We list some of them:

\begin{itemize}
\item $(\mathbb{Z}_p,+)$ is a compact topological group and $\mathcal{F}_p(\mathbb{N})=\mathcal{F}_p(\mathbb{Z})=\mathbb{Z}_p.$

\item Every $x\in \mathbb{Z}_p$ can be written as a power series $x=a_0+a_1p+a_2p^2+...$, where $a_i\in \{0,1,2,...,p-1\}.$


\item There is a probability measure $\mu_p$, the Haar measure, defined on the Borel subsets of $\mathbb{Z}_p$ such that $\mu_p(B_r(x))=r$, where $r=p^{-n}$ for some $n\ge 0$ and $B_r(x)= \{y\in \mathbb{Z}_p:\,\, |x-y|_p\leq r\}$. Sometimes we use the notation $B(x;r)$ to denote the ball $B_r(x)$.
\end{itemize}

Any $p$-adic integer can be associated to an infinite string $(a_0,a_1,...)$ in the set $\{0,1,...,p-1\}^{\mathbb{N}}$ and the measure $\mu_p$ coincides with the Bernoulli measure. Any ball $B_{p^{-n}}(x)=x(\mbox{mod.} \ p^n)+p^n\mathbb{Z}_p$ is in fact a cylinder. 

Throughout the paper we will write $\mathbb{U}=\{x\in \mathbb{Z}_p; |x|_p=1\}$. We observe that $\mathbb{U}$ is the set of units of $\mathbb{Z}_p$. It is easy to see that $p\mathbb{Z}_p=\mathbb{Z}_p\setminus \mathbb{U}$ is the unique maximal ideal of $\mathbb{Z}_p$, that is, $\mathbb{Z}_p$ is a local ring.

\subsection{Ergodic Theory on $\mathbb{Z}_p$}\label{Ergodic Theory}
Given a Borel measurable map $T:\mathbb{Z}_p\rightarrow \mathbb{Z}_p$, we say that a Borel measure $\mu$ is invariant by $T$ if $\mu(T^{-1}(A))=\mu(A)$ for all Borel sets $A$. We say that $(T,\mu)$ is an ergodic system if $\mu(A)=0$ or $\mu(A^c)=0$ for any set $A$ such that $\mu(A\Delta T^{-1}(A))=0$, where $A\Delta B=(A\setminus B) \cup (B\setminus A)$. A map $T:\mathbb{Z}_p\rightarrow \mathbb{Z}_p$ is uniquely ergodic if there exists only one probability invariant ergodic measure of it. 

 If $T:\mathbb{Z}_p\rightarrow \mathbb{Z}_p$ is given by $Tx=x+1$ then the Haar measure $\mu_p$ is the unique invariant ergodic probability measure of $T$. The map $T$ is the natural adding machine on $\mathbb{Z}_p$, see \cite{KM}.

More generally, Fan, Li, Yao and Zhou showed in \cite{FLYZ} that $T(x)=ax+b$ has only one invariant measure if $b\in \mathbb{U}$  and  $a\in B_{r_p}(1) \ \mbox{where} \ r_p=1, \mbox{ for} \ p\ge 3$ and $r_p=2, \ \mbox{if} \ p=2$  .

Let $\mathcal{M}(\mathbb{Z}_p)$ be the set of Borel probability measure on $\mathbb{Z}_p$. The weak$^{\ast}$-topology on $\mathcal{M} (\mathbb{Z}_p)$ is such that 

\begin{enumerate}
\item[a)] If $\mu_n\to \mu$ then $\liminf_n \mu_n(A)\ge \mu(A)$ for any open set $A$; 

\item[b)] If $\mu_n\to \mu$ then $\limsup_n \mu_n(F)\le \mu(F)$ for any closed set $F$. 
\end{enumerate}

In particular, if $\mu_n\to \mu$ in the weak$^{\ast}$ topology and if $F$ is a clopen subset of $\mathbb{Z}_p$ then $\liminf_n\mu_n(F)\ge \mu(F)\ge \limsup_n \mu_n(F)$, that is,
\begin{equation}\label{portmanteau}
    \lim_n\mu_n(F)=\mu(F).
\end{equation}
The equality (\ref{portmanteau}) in fact holds if $\mu_p(\partial F)=0.$ 

Fix a continuous map $T:\mathbb{Z}_p\rightarrow \mathbb{Z}_p$. This induces a continuous map $T_{\ast}:\mathcal{M}(\mathbb{Z}_p)\rightarrow \mathcal{M}(\mathbb{Z}_p)$ defined by $T_{\ast}\mu(S)=\mu(T^{-1}(S))$. In particular, $\mu$ is an invariant measure for $T$ if and only if $\mu$ is a fixed point of $T_{\ast}$. It is well know (see for instance \cite{KM}) that for every collection of intervals $\{I_n\}\subset \mathbb{N}$, with $\lim_n|I_n|=\infty$, any accumulation point of 
\begin{equation}\label{Portemanteau 2}
    \mu_n=\dfrac{1}{|I_n|}\sum_{j\in I_n}T^j_{\ast}\mu
\end{equation}
 is an invariant probability measure of $T$.
\subsection{Hausdorff Dimension and Box-Counting Dimension}

Let $\delta>0$. A $\delta$-cover $\mathcal{U}$ of a set $F\subset \mathbb{Z}_p$ is a family of balls $B_{r_i}(a_i)$, with $r_i<\delta$ for all $i$, such that $$ F\subset \cup_i B_{r_i}(a_i).$$ We define
	$$\mathcal{H}^s_{\delta}(F)=\inf_{\mathcal{U}} \left\{\sum_i r^s_i;  \  \mathcal{U}=\{B_{r_i}(a_i)\}_i \ \mbox{is a} \ \delta \text{-cover of} \ F\right\}.$$
The $s$-outer Hausdorff measure is given by
		$$\mathcal{H}^s(F)=\lim_{\delta\to 0} \mathcal{H}^s_{\delta}(F).$$
	\begin{definition}
	The Hausdorff dimension of $F$ is given by 
	$$HD(F)=\sup\{s>0; \mathcal{H}^s(F)=\infty\}=\inf \{s>0; \mathcal{H}^s(F)=0\}.$$
	\end{definition}
\begin{remark}
	If we write $s_0=HD(F)$, it is not guaranteed that $0<\mathcal{H}^{s_0}(F)<+\infty$. A set with this property is called \textit{regular} with respect to the Hausdorff measure.
\end{remark}

We point out to the reader that the definition of Hausdorff measure and Hausdorff dimension is easily adapted to any metric space.
	The Hausdorff dimension of $\mathbb{Z}_p$ is equal to 1. The $s$-Hausdorff measure has the following property which will be useful for us: Let $(X,d)$ be a metric space, and $S:X\rightarrow \mathbb{Z}_p$ a map such that 
$|S(x)-S(y)|_p\le Cd(x,y)$, then
		\begin{equation}\label{similaridade}\mathcal{H}^s(S(X))\le C^s\mathcal{H}^s(X).
        \end{equation}

\begin{definition}
     Consider $X$ a metric space. A closed set $F\subset X$ is $s$-Ahlfors-David regular if there exists a constant $C>0$ such that for every $x\in F$ one has
    \begin{equation}\label{ADset}
    C^{-1}\le \dfrac{\mathcal{H}^s(F\cap B_r(x))}{r^s}\le C.
    \end{equation}

    \begin{remark}
        We note that if $F$ is $s$-Ahlfors-David regular and $S$ is a bi-Lipschitz map then $S(F)$ is also $s$-Ahlfors-David regular.
    \end{remark}
\end{definition} 


\begin{definition}The \textit{upper} Box-counting dimension or just Box-dimension of a subset $F\subset \mathbb{Z}_p$ is given by
    $$BD(F)=\limsup_{n\to \infty} \dfrac{\log N_{p^{-n}}(F)}{n\log p},$$
    where $N_{p^{-n}}(F)$ is the minimum number of balls of radius $p^{-n}$ needed to cover $F$.
    We write 
    $$\underline{BD}(F)=\liminf_{n\to \infty}\dfrac{\log N_{p^{-n}}(F)}{n\log p}$$
    for the lower Box-counting dimension of $F\subset \mathbb{Z}_p$.
\end{definition}

\begin{example}
    Consider the set $F=\{1/n; n\in \mathbb{Z} \text{ and }\ {\rm gcd}(n,p)=1\}.$ Note that $F\subset \mathbb{Z}_p$ and $F$ is dense in $\mathbb{U}$. For every $n\in \mathbb{N}$ we have $N_{p^{-n}}(F)=p^n-p^{n-1}$ and then
        $$BD(F)=\limsup_{n\to \infty} \dfrac{\log p^{n-1}(p-1)}{n\log p}=1.$$
    \end{example}    

We remark that the above set has Box-dimension smaller than or equal to $1/2$ if we use the euclidean metric of $\mathbb{Q}$. Note that we also have that if $F\subset \mathbb{Z}_p$ then \begin{equation}
    BD(F)=BD(\mathcal{F}_p(F)
    ).
\end{equation}

\begin{remark}
For any set $F\subset \mathbb{Z}_p$ we have 
    $$HD(F)\le \underline{BD}(F)\le BD(F).$$
\end{remark}

From now on we write $G_n(F):=G_n=\{a\in \mathbb{Z}/p^n\mathbb{Z}; B_{p^{-n}}(a)\cap F\neq \emptyset\}$\footnote{Notice that, given $a\in \mathbb{Z}/p^n\mathbb{Z}$, the ball $B_{p^{-n}}(a)$ is well defined.}. The set of balls centered at the elements of $G_n$ form a natural cover of $F$ at the level $n$, that is, we look for $\mathbb{Z}_p=\cup_{j=0}^{p^n-1}B_{p^{-n}}(j)$ and try to find the set of $j$ for which $B_{p^{-n}}(j)$ has non-empty intersection with $F$. In this context we also use the notation $r_n=p^{-n}$.

    
    

	
	\section{Measure of integers and $p$-adic integers}
	
	\subsection{Measures on integers}
	Now we remind some notions of how we can measure subsets of integer numbers. 
	Given $E\subset \mathbb{Z}$ we define the $s$-counting measure of $E$ by 
		\begin{equation}\label{eq.H^}\mathcal{H}_{s}(E)=\limsup_{|I|\to \infty} \dfrac{|E\cap I|}{|I|^{s}},
		\end{equation}
	where $I$ runs over all intervals of integers numbers. We observe that 
	$$\mathcal{H}_1(E)=d^{\ast}(E)$$
	is the Banach density of $E$. 
	
	Moreover, the \textit{counting dimension} of $E$ is the number
		\begin{equation}\label{eq.D}
		D(E)=\limsup_{|I|\to \infty} \dfrac{\log |E\cap I|}{\log |I|}.
		\end{equation}
Below we remark some facts related to $\mathcal{H}_s$ and $D$.
\begin{remark}
We note that if $s < D(E)$ then $\mathcal{H}_{s}(E)=\infty$ and if $t > D(E)$ then $\mathcal{H}_{t}(F)=0.$ This property of $D(E)$ is similar to the one for the Hausdorff dimension. In the same way, if $s_0=D(E)$ we cannot say anything about $\mathcal{H}_{s_0}(E)$. If $0<\mathcal{H}_{s_0}(E)<+\infty$ the set $F$ is called \textit{regular} with respect to the counting measure.
\end{remark}	

We collect some basic properties of $D$

\begin{itemize}
    \item $E\subset F \ \implies \ D(E)\le D(F);$
    
    \item $D(E\cup F)=\max\{D(E),D(F)\}$.
\end{itemize}

\begin{example}
 
 Denote by $\mathbb{P}$ the set of prime numbers. We know that $d^*(\mathbb{P})=0$. However $D(\mathbb{P})=1$. This follows from The Prime Number Theorem. In fact, since $|\mathbb{P}\cap [1,n]|\sim \dfrac{n}{\log n}$ we have
 $$1\ge D(\mathbb{P})\ge \limsup_{n\to \infty}\dfrac{\log |\mathbb{P}\cap [1,n]|}{\log n}=\limsup_{n\to \infty}\dfrac{\log n - \log \log n}{\log n}=1.$$
\end{example}

\subsection{Natural maps}
In order to compare $d^{\ast}$, $\mathcal{H}_s$ and $D$ with $\mu_p$, $\mathcal{H}^s$ and $BD$ we consider two natural maps:

\begin{definition}
 If $E\subset \mathbb{Z}_p$ we define $\mathcal{F}_p:\mathcal{P}(\mathbb{Z}_p)\rightarrow \mathcal{P}(\mathbb{Z}_p)$ by $\mathcal{F}_p(E)=\overline{E}^p$, the $p$-closure of $E$. 
 On the other hand, we define $\Pi_p:\mathcal{P}(\mathbb{Z}_p)\rightarrow \mathcal{P}(\mathbb{Z})$ given by $\Pi_p(E)=E\cap \mathbb{Z}$. Moreover, $\Pi_p^{\ast}(F)=F\cap \mathbb{N}$.
\end{definition}

\begin{remark} If $E\subset \mathbb{Z}$, then we can see $E$ as a subset of $\mathbb{Z}_p$ and it makes sense to write $\mathcal{F}_p(E)$. This set is called the closed embedding of $E$ in $\mathbb{Z}_p$. The maps $\Pi_p$ and $\Pi^{\ast}_p$ are the projection of $F$ on $\mathbb{Z}$ and $\mathbb{N}$, respectively. When we want to study $E\subset \mathbb{Z}$ such that $E\cap \mathbb{Z}^{-}\neq \emptyset$  we consider $\Pi_p(\mathcal{F}_p(E))=\mathcal{F}_p(E)\cap \mathbb{Z}$ and if $E\cap \mathbb{Z}^{-}=\emptyset$ we consider $\Pi^{\ast}_p(\mathcal{F}_p(E))=\mathcal{F}_p(E)\cap \mathbb{N}.$ Thus, in the last case we just write $\Pi_p$ even for $\Pi_p^{\ast},$ that is, if $E\subset \mathbb{N}$ then $\Pi_p(\mathcal{F}_p(E))=\mathcal{F}_p(E)\cap \mathbb{N}.$
\end{remark}

\begin{example}\label{ex3}
 
Consider $\mathbb{P}$ the set of prime numbers. We shall find $\mu_p(\mathcal{F}_p(\mathbb{P}))$. Take $x\in \mathbb{U}$ and $\varepsilon>0$. Since $\mathbb{U}=\mathbb{Z}_p\setminus p\mathbb{Z}_p$ we have $$\mathbb{U}=\bigcup_{j=0; (j,p)=1}^{p^n-1} (j+p^n\mathbb{Z}_p)$$
for any $n\in \mathbb{N}$. In particular, for every $n$ there exists $j$ with $(j,p)=1$ such that  $x\in j+p^n\mathbb{Z}_p$. By Dirichlet's theorem about primes in arithmetic progression, the set $A=\{j+kp^n;k\in \mathbb{N}\}$ has infinitely many primes. If $q\in A$ is a prime number, then     $$|x-q|_p\le |p^n|=p^{-n}.$$
If we take $n$ such that $p^{-n}<\varepsilon$ then $|x-q|_p<\varepsilon$. This implies that the $p$-adic closure of $\mathbb{P}$ is $\mathbb{U}\cup \{p\}$.

Since $\mathcal{F}(\mathbb{P})=\mathbb{U}\cup \{p\}$ and $\mathbb{Z}_p=\mathbb{U}\cup p\mathbb{Z}_p,$ we have
    $$\mu_p(\mathcal{F}_p(\mathbb{P}))=1-1/p.$$
Since $d^{\ast}(\mathbb{P})=0$, in this case, we have $\mu_p(\mathcal{F}_p(\mathbb{P}))>d^{\ast}(\mathbb{P})$.    
 \end{example}
 
  \begin{example} Let $\phi: \mathbb{N}\rightarrow \mathbb{N}$ be the Euler's Totient function, given by, 
        $$\phi(n)=\#\{1\le j<n;\mbox{gcd}(j,n)=1\}.$$
    Let $\mathcal{V}=\phi(\mathbb{N})$ be the set of Totients. Since $\phi(\mathbb{P})=\mathbb{P}-1$ we have
    $$\mu_p(\mathcal{F}_p(\mathcal{V}))\ge 1-1/p.$$ 
    If $p=2$ then $\mu_2(\mathcal{F}_2(\mathcal{V}))=1/2,$ because $\mathcal{V}\cap (2\mathbb{N}+1)=\{1\}.$ 
    Suppose that $2$ is a primitive root modulus $p^2,$ which implies that $2$ is a primitive root modulus $p^n$ for every $n$. In this case, given $x\in \mathbb{U}$ and $\varepsilon>0$ take $n$ such that $p^{-n}<\varepsilon$. Let $0<j<p^n$ with $gcd(j,p)=1$ such that $x\in j+p^n\mathbb{Z}_p$. There is $m\in \mathbb{N}$ such that
        $$2^m\equiv j \ (\mbox{mod} \ p^n).$$
    In particular
        $$|x-\phi(2^{m+1})|_p\le |p^n|=p^{-n}<\varepsilon.$$
    This and the fact that $\mathcal{F}_p(\phi(\mathbb{P}))\supset p\mathbb{Z}_p$ implies that $\mathcal{F}_p(\mathcal{V})=\mathbb{Z}_p$.
    In particular, this happens for $p=3,5,7,11,13$. This method does not allow us to conclude the same for $p=17$ or $p=31$, for instance, because $2^8=256\equiv 1 \ (\mbox{mod} \ 17) $ and $2^5=32\equiv 1 \ (\mbox{mod} \ 31)$ which means that $2$ is not a primitive root modulus $p$.  
    \end{example}

\begin{example}
    Let $q\equiv 1 \ ( \mbox{mod.} \  p )$ be a prime. It is well known that the map $\varphi_q: \mathbb{Z}_p\rightarrow 1+p\mathbb{Z}_p$ given by $\varphi_q(z)=q^z$ is a homeomorphism, see for instance \cite{A}. This implies that $\mathcal{F}_p(\varphi_q(\mathbb{N}))=\varphi_q(\mathcal{F}_p(\mathbb{N}))=\varphi_q(\mathbb{Z}_p)=1+p\mathbb{Z}_p$. Note that we have
        $$HD(\mathcal{F}_q(\varphi_q(\mathbb{N})))=0$$ 
        but $$\mu_p(\mathcal{F}_p(\varphi_q(\mathbb{N})))=1/p, \ \mbox{which implies} \ HD(\mathcal{F}_p(\varphi_q(\mathbb{N}))=1.$$
    \end{example}
    The above example shows that the closed embedding may have a behaviour completely different for distinct values of primes $p\neq q$.

   \begin{example}
   If $E=B_{p^{-n}}(x)$ with $x\in \mathbb{Z}$ then $\Pi_p(E)$ is an arithmetic progression with common difference $p^n$ and containing $x$.
   \end{example}

   \subsection{Sets of integers which are projection of p-adic closed sets}\label{subsection nmap}
Remember the classes 

$$\mathcal{S}_p=\{E\subset \mathbb{Z}; E=\Pi_p(\mathcal{F}_p(E))=\mathcal{F}_p(E)\cap \mathbb{Z}\}$$
          and
            $$\mathcal{S}^{\ast}_p=\{E\subset \mathbb{N};E=\Pi_p(\mathcal{F}_p(E))=\mathcal{F}_p(E)\cap \mathbb{N}\}.$$
           First of all we relax the definition of these classes.
\begin{proposition}
    For any prime number $p$ we have $$\mathcal{S}_p=\{E\subset \mathbb{Z}; E=\Pi_p(F)=F\cap \mathbb{Z}, \ F \ \mbox{closed}\}$$ and $$\mathcal{S}^{\ast}_p=\{E\subset \mathbb{N}; E=\Pi_p(F)=F\cap \mathbb{N}, \ F \ \mbox{closed}\}.$$
\end{proposition}
           
\begin{proof}
    It is enough to prove that
    $$\{E\subset \mathbb{Z}; E=\Pi_p(F)=F\cap \mathbb{Z}, \ F \ \mbox{closed}\} \subset \mathcal{S}_p.$$
    Let $E\subset \mathbb{Z}$ such that there is a closed set $F\subset \mathbb{Z}_p$ with $E=\Pi_p(F).$ Since $E\subset \mathcal{F}_p(E)$ we have $E\subset \mathcal{F}_p(E)\cap \mathbb{Z}$. On the other side, if $E=F\cap \mathbb{Z}$ then $E\subset F$ and then $\mathcal{F}_p(E)\subset F$. This implies that $\mathcal{F}_p(E)\cap \mathbb{Z}\subset F\cap \mathbb{Z}=E$, that is, if $E=\Pi_p(F)$ for some $F$ closed then we can take $F=\mathcal{F}_p(E).$
\end{proof}

We write $E^{\ast}=E'\setminus E$ where $E'$ is the set of accumulation points of $E$. Since $\mathcal{F}_p(E)=E\cup E^{\ast}$, $E\in \mathcal{S}_p$ if and only if $E^{\ast}\cap \mathbb{Z}=\emptyset$, and $E\in \mathcal{S}^{\ast}_p$ if and only if $E^{\ast}\cap \mathbb{N}=\emptyset$.

\begin{proposition}\label{ch.1}
Consider $E\subset \mathbb{Z}$ an infinite arithmetic progression with common difference $r$. Then $E=\Pi_p(\mathcal{F}_p(E))$ if and only if $r=p^k$ for some $k$. 
\end{proposition}
\begin{proof}
Consider $E=\{a+rn; n\in \mathbb{Z}\}$. We will write $r=sp^k$ with $\mbox{gcd}(s,p)=1$. Take $z\in a+p^k\mathbb{Z}_p=B_{p^{-k}}(a)$. Write $z=a+p^k\tilde{z}$. Since $|s|_p=1$ we have $\dfrac{\tilde{z}}{s}\in \mathbb{Z}_p$. Given $m\in \mathbb{N}$ we choose $n\in \mathbb{N}$ such that
    $$\left | n - \dfrac{\tilde{z}}{s}\right |_p< p^{-m}.$$
    Take $N=a+sp^kn\in E$. Note that $|a+sp^kn-z|_p=|sp^kn-p^k\tilde{z}|_p$. In particular
    $$|a+sp^kn-z|_p=|s|_p|p^k|_p\left |n-\dfrac{\tilde{z}}{s}\right |_p<p^{-m}.$$ 
    This shows that $\mathcal{F}_p(E)=B_{p^{-k}}(a)$. Then
    $$\Pi_p(\mathcal{F}_p(E))=B_{p^{-k}}(a)\cap \mathbb{Z}=\{a+p^kn; n\in \mathbb{Z}\}=E \iff s=1.$$
\end{proof}

We remember that a $p$-adic number $z\in \mathbb{Q}_p$ is in $\mathbb{Q}$ if its $p$-adic expansion is eventually periodic, that is, if $z=\sum_{i=-M}^{\infty}a_ip^i$ then there exists $k\in \mathbb{Z}$ such that $a_{k+N}=a_N$ for $N$ large enough. This implies the following remark which we write as a lemma.

\begin{lemma}\label{caracterização}
A $p$-adic integer $z=\sum_{i\ge 0}a_ip^i$ is in $\mathbb{Z}$ if and only if there is $N\in \mathbb{N}$ such that either $a_i=0, \ i\ge N$ or $a_i=p-1, \ i\ge N$.
\end{lemma}
\begin{proof}
If $a_i=0$ for $i\ge N$  then $z=\sum_{i=0}^{N-1}a_ip^i\in \mathbb{Z}$. If $a_i=p-1, 
 i\ge N$ then $z=\sum_{i=0}^{N-1}a_ip^i+(p-1)p^N\sum_{j\ge 0}p^j$. Let $\tilde{z}=\sum_{i=0}^{N-1}a_ip^i$. Thus $z=\tilde{z}+(p-1)\dfrac{p^N}{1-p}=\tilde{z}-p^N\in \mathbb{Z}.$ 

 Conversely, suppose that $z=\sum_{i\ge 0}a_ip^i\in \mathbb{Z}$, in particular there is $k$ such that $a_{k+N}=a_N$ for $N$ large enough, we say $N\ge k_0$. Let $(c_0,c_1,...,c_{k-1})=(a_{N},a_{N+1},...,a_{k+N-1})$ for $N>k_0$ and $c=\sum_{i=0}^{k-1}c_ip^i$. We have
    $$z=\sum_{i=0}^{N-1}a_ip^i+\sum_{i\ge N}a_ip^i=\tilde{z}+p^N\cdot c \cdot \dfrac{1}{1-p^{k}}\ \in \mathbb{Z} \iff c\in \{0,p^k-1\}.$$
    Note that $c=0$ if and only in $c_i=0, \ i\in \{0,1,...,k-1\}$ which is equivalent to $a_i=0, \ i\ge N$. On the other hand, $c=p^k-1$ if and only if $c_i=p-1, \ i\in \{0,1,...,k-1\}$ which is equivalent to $a_i=p-1, \ i\ge N$.
\end{proof}

Consider $\gamma_j\in \{0,1,2,...,p-1\}^{r_j}, \ j=1,2,...,s$, $s$-strings with length $r_j$. We write $\mathcal{P}(\mathbb{Z};(\gamma_1,\gamma_2,...,\gamma_s)_p)$ the collection of sets of integers $E$ such that any element $x\in E$ does not exhibit any string $\gamma_j, \ j=1,2,...,s$ in its expansion at the base $p$.

\begin{lemma}\label{não enumerabilidade}
Let $p>3$ a prime number. For $\gamma_1=0$ and $\gamma_2=p-1$ the collection $\mathcal{P}(\mathbb{Z},(\gamma_1,\gamma_2)_p)$ is uncountable.
\end{lemma}
\begin{proof}
    Consider $\alpha=\{1,2,...,p-2\}^{\mathbb{N}_0}$, $\alpha=(a_i)_{i\in \mathbb{N}_0}$. We define $E(\alpha)\subset \mathbb{N}$ by $E(\alpha)=\{x_0,x_1,x_2,...\}$, where $x_n=\sum_{i=0}^na_ip^i$, $n\ge 0$. Note that if $\alpha\neq \beta$ then $E(\alpha)\neq E(\beta).$ In particular the map $E: \{1,2,...,p-2\}^{\mathbb{N}_0}\rightarrow \mathcal{P}(\mathbb{Z},(\gamma_1,\gamma_2)_p)$ is well defined and it is injective. Since $p>3$, $\{1,2,...,p-2\}^{\mathbb{N}_0}$ is an uncountable set  we have that $\mathcal{P}(\mathbb{Z},(\gamma_1,\gamma_2)_p)$ is also an uncountable set.
\end{proof}

\begin{remark}We observe that any set $E$ such that its elements have strings $\alpha=(p-1,...,p-1)$ or $\beta=(0,...,0)$ with length bounded by $M>0$ then there is no element in $E^{\ast}\cap \mathbb{Z}=\emptyset$, that is, $\Pi_p(\mathcal{F}_p(E))=E$. Consider $g:\mathbb{Z}_p\rightarrow \mathbb{Z}_p$ given by $g(x)=\dfrac{x-a_0(x)}{p}$, where $x=\sum_{i\ge 0}a_i(x)p^i$. Since $(g,\mu_p)$ is an ergodic system using our Theorem A, every set in this class has Banach density zero. However, we can easily obtain sets with positive Banach density such that $\Pi_p(\mathcal{F}_p(E))=E$ by Proposition \ref{ch.1}. Moreover, we have $d^{\ast}(E)=\mu_p(\mathcal{F}_p(E))$ because for every set with these properties mentioned above is such that $\mu_p(\partial(\mathcal{F}_p(E)))=0$
\end{remark}
Next we establish a consequence of Lemma \ref{caracterização} and Lemma \ref{não enumerabilidade}.
\begin{proposition}
Consider $p>3$. The collection $\mathcal{S}_p$ is uncountable.
\end{proposition}
\begin{proof}
By Lemma \ref{caracterização} we have that $\mathcal{P}(\mathbb{Z},(\gamma_1,\gamma_2)_p)\subset \mathcal{S}_p$ and by Lemma \ref{não enumerabilidade} we have that $\mathcal{P}(\mathbb{Z},(\gamma_1,\gamma_2)_p)$ is uncountable.
\end{proof}

Let $\Sigma^{k}=\{0,1,2,...,p-1\}^{-k}=\{(a_k,a_{k-1},...,a_1,a_0); a_i\in\{0,1,2,...,p-1\}\}$, $\Sigma=\cup_{k\ge 0} \Sigma^{k}$ the disjoint union of $\Sigma^{k}$. On the $\Sigma$ we define the shift map
$$\sigma((a_k,a_{k-1},a_{k-2},...,a_2,a_1,a_0))=(a_{k-1},a_{k-2},...,a_2,a_1,a_0), k>0$$  and
$\sigma(a)=0, \ a\in\{0,1,...,p-1\}.$
Note that the map $\sigma$ can be seen as a map on the natural numbers, since every $n\in \mathbb{N}$ has a $p$-adic finite expansion. If $n=(a_k,a_{k-1},...,a_1,a_0)_p=\sum_{i=0}^ka_ip^i$ and $k>0$ then $\sigma(n)=n-a_kp^k=(a_{k-1},...,a_0)_p$ where $a_{k-1}=0$ is allowed.


\begin{lemma}\label{inv1}
    Let $E\subset \mathbb{N}$ be a set $\sigma$-invariant, that is, $\sigma(E)\subset E$. Then $E=\Pi_p(\mathcal{F}_p(E))=\mathcal{F}_p(E)\cap \mathbb{N}.$
\end{lemma}
\begin{proof}
Note that if $x\in E^{\ast}\cap \mathbb{N}\neq \emptyset$ then the infinite $p$-adic expansion of $x$ is the form $(...,0,0,a_k,a_{k-1},...,a_1,a_0)$ and if $x_n\in E$ goes to $x$ we have $x_n=(\underline{b},0_s,a_k,a_{k-1},...,a_1,a_0)$ for $n$ arbitrarily large, where $\underline{b}\in \Sigma^{m}$ for some $m$ and $0_s=(0,0,...,0)$ with $s$-zeroes. Since $\sigma(E)\subset E$ we have $\sigma^{m+s}(x_n)=x\in E$. 
\end{proof}

\begin{remark}
We note that it is possible $\mathcal{F}_p(E)\cap \mathbb{Z}\supsetneq E$ because if $x\in E^{\ast}\cap \mathbb{Z}$ is negative, its $p$-adic expansion is of the form $(...,p-1,p-1,...,p-1,a_k,a_{k-1},...,a_1,a_0)$ and if $E$ contains elements such as $((p-1)_s,a_k,a_{k-1},...,a_1,a_0)$ for infinitely many values of $s$ with a fixed $k$ then these elements go to $x$, but $x\notin E$. 
\end{remark}

\begin{example}
    Consider $E=\{x_1,x_2,...\}\subset \mathbb{N}$ where $x_j=\sum_{i=0}^{j-1}(p-1)p^{i}, j\ge 1$. Note that $\sigma(E)\subset E$ and $\Pi_p(\mathcal{F}_p(E))=\mathcal{F}_p(E)\cap \mathbb{N}=E$. We have $\mathcal{F}_p(E)\cap \mathbb{Z}=E\cup \{-1\}$
\end{example}
In order to given a sufficient condition on the set $E\subset \mathbb{Z}$ with $E\cap \mathbb{Z}^{-}$ is infinite we give the following definition. 
    \begin{definition}
        The $p$-adic dual of an integer number $n$ is defined in the following way
        \begin{itemize}
        \item $\overleftarrow{n}=n-p^N$ if $a_i=0, \ \forall i\ge N$ and $a_{N-1}\neq 0$;
        \item $\overleftarrow{n}=n+p^N$ if $a_i=p-1, \ \forall i\ge N$ and $a_{N-1}\neq p-1$.
        \end{itemize}
        Moreover, for $n=0$ we define
        \begin{itemize}
        \item $\overleftarrow{0}=-1$.
        \end{itemize}
    \end{definition}
We write $\overleftarrow{E}_p=\{\overleftarrow{n}; n\in E\}$ the $p$-adic dual set of $E$. The proposition below gives a sufficient condition for a set $E\subset \mathbb{Z}$ under which $E\in \mathcal{S}_p.$

    \begin{proposition}
        Consider $E\subset \mathbb{Z}$. If $E\cap \mathbb{N}$ is $\sigma$-invariant and $\overleftarrow{E}_p \subset E$ then $\Pi_p(\mathcal{F}_p(E))=\mathcal{F}_p(E)\cap \mathbb{Z}=E.$ 
    \end{proposition}
    
\begin{proof}
    In fact, consider $x=(...,p-1,...,p-1,a_k,a_{k-1},...,a_1,a_0)\in E^{\ast}\cap \mathbb{Z}^{-}.$ First we suppose $x_n\in E\cap \mathbb{N}$ with $x_n\to x$.  Then, we can write  $x_n=(\underline{b}_n,p-1,...,p-1,a_k,a_{k-1},...,a_1,a_0)\in \mathbb{N}$ where the string $(p-1,...,p-1)$ has at least $n$ numbers $p-1$ and $a_k=p-1$. In particular, since $E\cap \mathbb{N}$ is $\sigma$-invariant $z=(a_k,a_{k-1},...,a_1,a_0)\in E$. But $x=a_0+a_1p+...+a_{k-1}p^{k-1}+a_kp^k-p^{k+1}=\overleftarrow{z}$. That is, $x\in \overleftarrow{E}_p$. Next, we suppose $x_n\in E\cap \mathbb{Z}^{-}$ with $x_n\to x$. The $p$-adic expansion of a negative number has the infinite block $(...,p-1,p-1,...,p-1)$. Then we can write $x_n=y_n-p^{k_n}$ for some $k_n>0$ and $a_{k_n-1}(y)\neq p-1$. Thus, $y_n=\overleftarrow{x_n}\in \mathbb{N}$ and since $\overleftarrow{E}_p \subset E$ we have $y_n\in E$. Note that the sequence $k_n$ is unbounded otherwise $\{x_n\}_n$ would be finite. Then, $y_n\to x$ and by the first part $x\in \overleftarrow{E}_p \subset E$. If we take $x\in E^{\ast}\cap \mathbb{N}$ the statement follows from Lemma \ref{inv1}.
    \end{proof}

    We postpone to the next section the result which establishes a necessary condition for a set $E\subset \mathbb{Z}$, respectively $E\subset \mathbb{N}$ under which $E\in \mathcal{S}_p$, respectively $\mathcal{S}^{\ast}_p$.

\subsection{Local measures on p-adic integers}\label{subs3.4}
\bigskip


Next we construct local measures of a $p$-adic set $F$ which are indeed local densities with respect to the counting measure as we shall see in the Proposition \ref{prop.3.1}.	More precisely, remember that given $F\subset \mathbb{Z}_p$ and  $n\in \mathbb{N}$, we define
    $$\eta_s(F;x,n)=\limsup_{|I|\to \infty}\dfrac{|F\cap B_{p^{-n}}(x)\cap I|}{|B_{p^{-n}}(x)\cap I|^s}.$$



\begin{definition}
The $s$-counting local measure of $F\subset \mathbb{Z}_p$ at $x$ is the number
$$\eta_s(F,x)=\limsup_{n\to \infty}\eta_s(F;x,n).
$$
\end{definition}

\begin{remark}
The function $\eta_s(F;\cdot,n): \mathbb{Z}_p\rightarrow [0,\infty]$ is locally constant; in fact, if $y\in B_{p^{-n}}(x)$ then $B_{p^{-n}}(y)=B_{p^{-n}}(x),$ which implies $\eta_s(F;x,n)=\eta_s(F;y,n)$. In particular, $\eta_s(F;\cdot)$ is a measurable function with respect to the Borel $\sigma$-algebra of $\mathbb{Z}_p$.
\end{remark}

Remember that given a set $\mathcal{F}\subset \mathcal{P}(\mathbb{Z})$ and $s>0$, we say that $E\subset \mathbb{Z}$ is $s$-abundant in $\mathcal{F}$ if there exists a sequence of finite sets $\{A_j\}_{j\in \mathbb{N}}\subset \mathcal{F}$ and $c>0$ such that $|A_{k+1}|>|A_k|$ and
       $$|E\cap A_k|\ge c\cdot |A_k|^s.$$

Remember also that $\mathcal{A}^{r}_f\subset \mathcal{P}(\mathbb{Z})$ denotes the set of all finite arithmetic progressions of common difference $r$.

In the next proposition we will see that we can obtain number theoretic information from the function $\eta_s$.

\begin{proposition}\label{etas:teorianumeros}
Given $F\in \mathbb{Z}_p$, if there is a point $x\in \mathbb{Z}_p$ such that $\eta_s(F;x)>0$, then for an unbounded sequence $\{n_j\}_j$ of natural numbers we have that $\Pi_p(F)\subset \mathbb{Z}$ is $s$-abundant in $\mathcal{A}^{p^{n_j}}_f$. 
\end{proposition}

\begin{proof}
Take $x\in F$ such that $\eta_s(F;x)=\delta>0.$ In particular there exists an unbounded sequence $\{n_j\}_j$ such that 
$$\limsup_{|I|\to +\infty}\dfrac{|F\cap B_{p^{-n_j}}(x)\cap I|}{|B_{p^{-n_j}}(x)\cap I|^s}>\dfrac{\delta}{2}.$$
For any $j$ we can take $\{I^j_m\}_m$ sequence of intervals with $|I^j_{m+1}|>|I^j_m|$ and 
$$|F\cap B_{p^{-n_j}}(x)\cap I^j_m|>\dfrac{\delta}{4}|B_{p^{-n_j}}(x)\cap I^j_m|^s.$$
The sets $A^j_m=B_{p^{-n_j}}(x)\cap I^j_m$ are arithmetic progressions of common difference $p^{n_j}$ and
$$|F\cap A^j_m|>\dfrac{\delta}{4}|A^j_m|^s.$$
\end{proof}

We start our results with the following Lemma. Since the proof is just to count the number of elements in the arithmetic progression $B_{p^{-n}}(x)\cap I$ we omit it.
\begin{lemma}\label{lem1}
Consider an interval $I\subset \mathbb{Z}$. Let $a\in \mathbb{Z}$ and $r_n=p^{-n}$, then

	$$|B_{r_n}(a)\cap I|=\left \lfloor |I| \cdot r_n\right \rfloor \ \mbox{or} \ |B_{r_n}(a)\cap I|=\left \lfloor |I| \cdot r_n \right \rfloor + 1.$$
\end{lemma}



Let us give an example

\begin{example}
 In $\mathbb{Z}_2$ we have
    $$B_{2^{-n}}(1)\cap (M,N]=\left \{1+\left(\left \lfloor\dfrac{M-1}{2^n}\right \rfloor+1\right)2^n,...,1+\left \lfloor\dfrac{N-1}{2^n}\right \rfloor 2^n\right \}.$$
In particular, if $M=2^{nj}+1$ and $N={2^{nj+1}}$ then
    $$B_{2^{-n}}(1)\cap (M,N]= \{1+2^{nj}+2^n, ...,1+2^{nj+1}-2^n \}.$$
    If $E=\{1+2m; m\in \mathbb{N}\}$ then 
    $\overline{E}^{2}\cap B_{2^{-n}}(1)\cap (M,N]=B_{2^{-n}}(1)\cap (M,N]$.
    Therefore, for $s=1$
    $$\eta_s(\bar{E}^2;1,n)=\limsup_{|I|\to \infty}\dfrac{|\overline{E}^{2}\cap B_{2^{-n}}(1)\cap I|}{|B_{2^{-n}}(1)\cap I|}=1$$
    does not depends on $n$.
\end{example}

Let $F\subset \mathbb{Z}_p$, we note that 
$$\dfrac{|F\cap B_{p^{-n}}(x)\cap I|}{|B_{p^{-n}}(x)\cap I|^s}=\dfrac{|F\cap B_{p^{-n}}(x)\cap I|}{|I|^{s}}\cdot \dfrac{|I|^s}{|B_{p^{-n}}(x)\cap I|^s},$$
when $|I|$ goes to infinity we have
$$\eta_s(F;x,n) =\dfrac{\mathcal{H}_s(F\cap B_{p^{-n}}(x)\cap \mathbb{Z})}{d^{\ast}(B_{p^{-n}}(x)\cap \mathbb{Z})^s}= \dfrac{\mathcal{H}_s(F\cap B_{p^{-n}}(x)\cap \mathbb{Z})}{p^{-ns}}.$$
In particular we have the following proposition

\begin{proposition}\label{prop.3.1} We have
   $$\eta_s(F,x)= \limsup_{n\to \infty} \dfrac{ \mathcal{H}_s(F\cap B_{p^{-n}}(x)\cap \mathbb{Z})}{p^{-ns}}$$
\end{proposition}

\begin{remark}
    The Proposition \ref{prop.3.1} shows that $\eta_s$ is a local density associated to $\mathcal{H}_s$.
\end{remark}

\begin{corollary}
 If $t>D(F\cap \mathbb{Z})$ then $\eta_t(F;x)=0$ for every $x\in \mathbb{Z}_p$.
\end{corollary}
\begin{proof}
By Proposition \ref{prop.3.1} 
    $$\eta_t(F;x)=\limsup_{n\to \infty}\dfrac{\mathcal{H}_t(F\cap B_{p^{-n}}(x)\cap \mathbb{Z})}{p^{-nt}}=0, \ \forall \ x,$$
since $\mathcal{H}_t(F\cap B_{p^{-n}}(x)\cap \mathbb{Z})\le H_t(F\cap \mathbb{Z})=0$, if $t>D(F\cap \mathbb{Z}).$    
    
\end{proof}



    \section{$\mathcal{H}_s$ vs $\mathcal{H}^s$ results}\label{3}


In this section we prove our main results. We start by showing that there is a relation between the $p$-adic Haar measure of $F\subset \mathbb{Z}_p$ and the upper Banach density of $\Pi_p(F)\subset \mathbb{Z}$. Then we generalize this result to other dimensions, changing the Haar measure by the $s$-dimensional Hausdorff measure and the upper Banach density by $s$-counting measure. We give two proofs for the theorem below, a combinatorial measure-theoretical proof, using Lemma \ref{lem1} and an ergodic-theoretical proof. Here there is no difference take $\Pi_p(F)=F\cap \mathbb{N}$ or $\Pi_p(F)=F\cap \mathbb{N}$ the proofs can easily adapted.

\begin{mydef}
If $F\subset \mathbb{Z}_p$ is closed, then
	$$ d^{\ast}(\Pi_p(F))\le \mu_p(F).$$
 Moreover, if $\mu_p(\partial(F))=0$ then the equality holds.
\end{mydef}
\begin{proof}
Let $r_n=p^{-n}$ and write $G_n=\{a\in \mathbb{Z}/p^n\mathbb{Z}; B_{r_n}(a)\cap F\neq \emptyset\}$. Note that

	$$F=\bigcap_{n\ge 0}\bigcup_{a\in G_n} B_{r_n}(a).$$
	
If we denote $\Delta_n(F)=\bigcup_{a\in G_n} B_{r_n}(a)$, then $\Delta_{n+1}(F)\subset \Delta_{n}(F)$ and one has 
	$$\mu_p(F)=\lim_{n\to \infty} \mu_p(\Delta_n(F)).$$	
	On the other hand, 
		$$\mu_p(\Delta_n(F))=\sum_{a\in G_n} \mu_p(B_{r_n}(a))= |G_n| \cdot r_n. $$
	Since $F\cap I \subset \Delta_n(F)\cap I$ we have 
		$$|F\cap I|\le \sum_{a\in G_n} |B_{r_n}(a)\cap I|\le \sum_{a\in G_n} (r_n |I|+1)=|G_n|\cdot r_n \cdot |I|+|G_n|,$$
	where we used Lemma \ref{lem1}. In particular,
		$$d^*(F\cap \mathbb{Z})\le |G_n| \cdot r_n.$$
	taking $n\to \infty$	we have		
		$$d^*(F\cap \mathbb{Z})\le \mu_p(F).$$
  

We also give another proof. Let $I_n\subset \mathbb{Z}$ be a family of intervals of integers numbers, with $\lim |I_n|=\infty$. Let $\delta_0$ be the Dirac measure supported in $0\in \mathbb{Z}_p$ and consider the family of measures $\displaystyle \nu_n=\dfrac{1}{|I_n|}\sum_{j\in I_n} T_{\ast}^j\delta_0$, where $T(x)=x+1$. Since $T$ is uniquely ergodic, we have in the weak$^*$-topology 
    $$\lim_n \nu_n=\mu_p.$$
    For any closed set $F\subset \mathbb{Z}_p$ one has $\limsup_n \nu_n(F)\le \mu_p(F)$, that is,
    \begin{equation}\label{Portmanteau}
        d^*(F\cap \mathbb{Z})\le \mu_p(F).
    \end{equation}
    Moreover, if $F$ is such that $\mu_p(\partial(F))=0$ then (\ref{portmanteau})  implies equality in (\ref{Portmanteau}). 

    \end{proof}	

We are ready to prove the Corollary \ref{cor.1.1}. 

\begin{flushleft}
\textbf{Corollary 1.1.} Consider $E, F\subset \mathbb{Z}$ such that $d^{\ast}(E)+d^{\ast}(F)>1$, then for every $p$ prime and for every $n
 \in \mathbb{N}$ the arithmetic difference $E-F$ contains multiples of $p^n$. 
\end{flushleft}

\begin{proof}
If $E,F\subset \mathbb{Z}$ are such that $d^{\ast}(E)+d^{\ast}(F)>1$, then by Proposition A $$\mu_p(\mathcal{F}_p(E))+\mu_p(\mathcal{F}_p(F))>1,$$ for any prime $p\ge 2.$ In particular, 
$\mu_p(\mathcal{F}_p(E)\cap \mathcal{F}_p(F))>0.$ Take $x\in \mathcal{F}_p(E)\cap \mathcal{F}_p(F)$. For any $k\in \mathbb{N}$, there are $m\in E$ and $n\in F$ such that 
    $$|x-m|_p\le p^{-k} \ \mbox{and} \ |x-n|_p\le p^{-k}.$$
But this implies $|m-n|_p\le p^{-k}$, that is, $m\equiv n \ (\mbox{mod } p^k).$    
\end{proof}

\begin{remark}
The equality $d^{\ast}(E)+d^{\ast}(F)=1$ is not enough to conclude the statement. For instance, take $E=\{2n;n\in \mathbb{N}\}$ and $F=\{2n+1; n\in \mathbb{N}\}$. We have $d^{\ast}(E)+d^{\ast}(F)=1$ but $m-n\equiv 0 \ (\mbox{mod.} \ 2^k)$ never holds for any $k>0$, $m\in E$ and $n\in F$.
\end{remark}
		We remember the definition of upper and lower density of a set $F$.
	\begin{definition}
	    Let $F\subset \mathbb{Z}_p$. The upper $s$-density of $F$ at $x$ are respectively
	    $$\overline{\theta}^s(F;x)=\limsup_{n\to \infty} \dfrac{\mathcal{H}^s(F\cap B_{p^{-n}}(x))}{p^{-ns}},$$
	\end{definition}
It is not difficult to see that if $F$ is closed then $\overline{\theta}^s(F;x)=0$ if $x\notin F$. Now we show that the $s$-upper local density is almost always bounded from below by 1.

\begin{remark}\label{ADset}
    We note that if $F$ is an s-Ahlfors-David regular set then there exists $C>0$ such that
    \begin{equation}
    C^{-1}\le \overline{\theta}^s(F;x)\le C, \ \forall \ x\in F. 
    \end{equation}
\end{remark}

\begin{proposition} \label{prop.4.1}
   Let $F\subset \mathbb{Z}_p$ be a measurable set with $\mathcal{H}^s(F)<+\infty$. For $\mathcal{H}^s$-almost every $x\in F$ we have $\overline{\theta}^s(F;x)= 1.$
\end{proposition}
	
\begin{proof} Firstly we prove that $\overline{\theta}^s(F;x)\ge 1$ for $\mathcal{H}^s$-almost every $x\in F$.  For $c>0$ define $F_c=\{x\in F; \overline{\theta}^s(F;x)< c \}$, and for $\delta>0$ let $\tilde{F}_{\delta}$ be the set 
    $$\tilde{F}_{\delta}=\{x\in F_c;\mathcal{H}^s(F\cap B_{p^{-n}}(x))<cp^{-ns} \text{ if } p^{-n}<\delta\}.$$
Let $\mathcal{U}$ be a $\delta$-cover of $F_c$ by balls $B_i$ of radius $p^{-n_i},$ in particular $\mathcal{U}$ is also a cover of $\tilde{F}_{\delta}$. For each $x\in B_i\cap \tilde{F}_{\delta}$ we have $B_i=B_{p^{-n_i}}(x).$ Thus, by the definition of $\tilde{F}_{\delta}$
    $$\mathcal{H}^s(\tilde{F}_{\delta}\cap B_i)=\mathcal{H}^s(\tilde{F}_{\delta}\cap B_{p^{-n_i}}(x))\le \mathcal{H}^s(F\cap B_{p^{-n_i}}(x)) <cp^{-n_i\cdot s}.$$
Then  
    $$\mathcal{H}^s(\tilde{F}_{\delta})\le \sum_{i;B_i\cap \tilde{F}_{\delta}\neq \emptyset} \mathcal{H}^s(\tilde{F}_{\delta}\cap B_i)< c\sum_{i;B_i\cap F_c\neq \emptyset} p^{-n_i\cdot s},$$
Since $\mathcal{U}$ is arbitrary we get
    $$\mathcal{H}^s(\tilde{F}_{\delta}) \le c \mathcal{H}^s(F_c).$$
Note that
    $$F_c=\bigcup_{\delta>0} \tilde{F}_{\delta}$$
and then 
    \begin{equation}\label{eq.4.1}\mathcal{H}^s(F_c)\le c \mathcal{H}^s(F_c).
    \end{equation}
If $c<1$ then the inequality (\ref{eq.4.1}) only happens if $\mathcal{H}^s(F_c)=0$. Denote $c_n=1-\dfrac{1}{n}.$ Note that 
        $$F_1=\bigcup_{n\ge 1} F_{c_n}$$
    and then $\mathcal{H}^s(F_1)=0$. Since
    $F=F_1\cup (F\setminus F_1)$ we conclude that 
        $$\overline{\theta}^s(F;x)\ge 1, \text{ for }\ \mathcal{H}^s- \text{almost all} \ x\in F.$$
The other inequality is not difficult, it follows of ideas from \cite{Falconer} prop 5.1 and we omit here.

\end{proof}	

Theorem A is in fact a consequence of the following result.

   \begin{lemma}
Let $F\subset \mathbb{Z}_p$ be a closed set and suppose that there is a constant $C>0$ such that
\[\limsup_{|I|\to +\infty}\frac{|F\cap B\cap I|}{|B\cap I|^s} \le C,\]
for any ball $B\subset \mathbb{Z}_p$. Then we have
\[\mathcal{H}_s(\Pi_p(F))\le C \mathcal{H}^s(F).\]
\end{lemma}

\begin{proof}
If $\mathcal{H}^s(F)=\infty$ then we have nothing to do. Suppose $\mathcal{H}^s(F)<\infty$ and choose $\epsilon>0$, we can find balls $B_1,....,B_n$ covering $F$ such that (here one uses that $F$ is compact)
\[\sum_{i=1}^n \text{diam}(B_i)^s \le \mathcal{H}^s(F)+\epsilon.\]
We have then

\begin{align*}
\limsup_{|I|\to \infty}\frac{|F\cap I|}{|I|^s} &\leq \sum_{i=1}^n \limsup_{|I|\to \infty}\frac{|F\cap B_i \cap I|}{|I|^s}\\
&= \sum_{i=1}^n \limsup_{|I|\to \infty}\frac{|F\cap B_i \cap I|}{|B_i\cap I|^s} \cdot \left(\frac{|B_i \cap I|}{|I|}\right)^s\\
&\le C \sum_{i=1}^n \text{diam}(B_i)^s \le C \mathcal{H}^s(F)+C\epsilon.
\end{align*}

Taking $\epsilon$ going to zero we conclude that $\mathcal{H}_s(F)\le C \mathcal{H}^s(F)$.

\end{proof}
We observe that whenever $s=1$ the constant $C$ can be taken equal to $1$ and we have Theorem A.
\begin{mydef2}
Let $F\subset \mathbb{Z}_p$ be a closed set and assume $\mathcal{H}^s(F)<\infty$. Suppose that there is a constant $C>0$ such that
\[\limsup_{|I|\to \infty}\frac{|F\cap B\cap I|}{|B\cap I|^s} \le C,\]
for any ball $B\subset \mathbb{Z}_p$. Then we have
\[\mathcal{H}_s(\Pi_p(F))\le \int_{F} \eta_s(F;x) d\mathcal{H}^s(x).\]
\end{mydef2}

\begin{proof}
We know that
\[\limsup_{n\to \infty} \frac{\mathcal{H}^s(F\cap B_{p^{-n}}(x))}{p^{-ns}}=1\]
for all $x\in F$, except for a set $F_0\subset F$ of $\mathcal{H}^s$-measure zero. Choose $\epsilon >0$. For $m\in \mathbb{N}$ consider 
$$A_m=\left \{x\in F; \limsup_{|I|\to \infty} \frac{|F\cap B_{p^{-n}}(x)\cap I|}{|B_{p^{-n}}(x)\cap I|^s} > \eta_s(F;x) + \epsilon\ , \ \mbox{for some} \ n>m \right \}.$$
Then $A_m\supset A_{m+1}$ and $\cap_m A_m=\emptyset$. Fix $n_0$ such that $\mathcal{H}^s(A_{n_0})<\epsilon$. We rewrite $A_{n_0}:=F_1$.

For each $x\in F\setminus (F_0\cup F_1)$ we can choose a ball $B(x)$, containing $x$, such that $|B(x)|<p^{-n_0}$ and
\[\frac{\mathcal{H}^s(F\cap B(x))}{\text{diam}(B(x))^s} \ge 1-\epsilon.\]
We can then select a countable family of disjoint balls $B_1,\,B_2,...$ covering $F\setminus (F_0\cup F_1)$ and such that $\text{diam}(B_i)<p^{-n_0}$ and
\[\frac{\mathcal{H}^s(F\cap B_i)}{\text{diam}(B_i)^s} \ge 1-\epsilon,\]
for all $i$. Let $n_1$ such that 
\[\mathcal{H}^s(F\setminus (F_0\cup F_1\cup B_1\cup...\cup B_{n_1}))<\epsilon.\]

Define $F_2 = F\setminus (B_1 \cup...\cup B_{n_1})$, one has that $F_2$ is closed and $\mathcal{H}^s(F_2)\le 2\epsilon$. Using the previous lemma one gets
\begin{align*}
\mathcal{H}_s(\Pi_p(F))&\le \mathcal{H}_s(\Pi_p(F\cap (B_1 \cup...\cup B_{n_1})))+\mathcal{H}_s(\Pi_p(F_2)) \\ &\le \mathcal{H}_s(\Pi_p(F\cap (B_1 \cup...\cup B_{n_1})))+2\epsilon C.
\end{align*}
On the other hand
\begin{align*}
\mathcal{H}_s(\Pi_p(F\cap (B_1 \cup...\cup B_{n_1})))&\le \sum_{i=1}^{n_1} \limsup_{|I|\to \infty}\frac{|F\cap B_i \cap I|}{|I|^s}\\
&= \sum_{i=1}^{n_1} \limsup_{|I|\to \infty}\frac{|F\cap B_i \cap I|}{|B_i\cap I|^s} \cdot \left(\frac{|B_i \cap I|}{|I|}\right)^s\\
&= \sum_{i=1}^{n_1} \limsup_{|I|\to \infty}\frac{|F\cap B_i \cap I|}{|B_i\cap I|^s} \cdot \text{diam}(B_i)^s\\
&= \sum_{i=1}^{n_1} \limsup_{|I|\to \infty}\frac{|F\cap B_i \cap I|}{|B_i\cap I|^s} \cdot \mathcal{H}^s(F\cap B_i)\cdot \frac{\text{diam}(B_i)^s}{\mathcal{H}^s(F\cap B_i)}\\
&\le \frac{1}{1-\epsilon}\sum_{i=1}^{n_1} \limsup_{|I|\to \infty}\frac{|F\cap B_i \cap I|}{|B_i\cap I|^s} \cdot \mathcal{H}^s(F\cap B_i).\\
\end{align*}

Since $\mathcal{H}^s(F\cap B_i)=\left( \mathcal{H}^s(F\cap B_i\setminus F_1) + \mathcal{H}^s(F_1\cap B_i) \right)$ we have
\begin{align*}
\displaystyle\limsup_{|I|\to \infty}\frac{|F\cap B_i \cap I|}{|B_i\cap I|^s} \cdot \mathcal{H}^s(F\cap B_i) \le \int_{F\cap B_i\setminus F_1} (\eta_s(F;x)+ \epsilon) d\mathcal{H}^s(x) + C \cdot \mathcal{H}^s(F_1\cap B_i).
\end{align*}
Summing over $i$ one concludes that
\begin{align*}
\mathcal{H}_s(\Pi_p(F\cap (B_1 \cup...\cup B_{n_1})))&\le \frac{1}{1-\epsilon}\sum_{i=1}^{n_1} \limsup_{|I|\to \infty}\frac{|F\cap B_i \cap I|}{|B_i\cap I|^s} \cdot \mathcal{H}^s(F\cap B_i)\\
&\le \frac{1}{1-\epsilon} \left[\int_{F} (\eta_s(F;x)+ \epsilon) d\mathcal{H}^s(x) + C \cdot \mathcal{H}^s(F_1)\right]\\
&\le \frac{1}{1-\epsilon} \left[\int_{F} \eta_s(F;x) d\mathcal{H}^s(x) + \epsilon \mathcal{H}^s(F)+ C \epsilon \right].
\end{align*}
Therefore
\begin{align*}
\mathcal{H}_s(\Pi_p(F))&\le \mathcal{H}_s(\Pi_p(F\cap (B_1 \cup...\cup B_{n_1})))+2\epsilon C  \\  &\le \frac{1}{1-\epsilon} \left[\int_{F} \eta_s(F;x) d\mathcal{H}^s(x) + \epsilon \mathcal{H}^s(F)+ C \epsilon \right]+2\epsilon C.
\end{align*}
Taking $\epsilon$ going to zero one obtains
\[\mathcal{H}_s(F)\le \int_{F} \eta_s(F;x) d\mathcal{H}^s(x).\]
\end{proof}
For each $n\in \mathbb{N}$ we can look for $\underline{\theta}^s_n(F)=\min_{x\in G_n} \dfrac{\mathcal{H}^s(F\cap B_{p^{-n}}(x))}{p^{-ns}}$ and we define 

\begin{equation}\label{dM}\underline{\theta}^s_{\ast}(F)=\limsup_{n\to \infty} \underline{\theta}^s_n(F).
\end{equation}
 If $F$ is a $s$-Ahlfors-David regular set then there exists $C>0$ such that (\ref{ADset}) occurs. In particular, 
\begin{equation} \label{ADset1}   
C^{-1}\le\underline{\theta}^s_{\ast}(F)\le C.
\end{equation}


In the following result we remove the hypotheses $\eta_s(F;x,n)\le C, \ \forall x, \ n\in \mathbb{N}$ but this implies that we need to compensate with a constant in the term of $\mathcal{H}_s.$ Precisely 

\begin{mydef3}\label{prop.4.2}
Let $F\subset \mathbb{Z}_p$ be a closed set such that $\mathcal{H}^s(F)< \infty$ and assume that there is an $\mathcal{H}^s$-integrable function $g$ such that $\eta_s(F;x,n)\le g(x)$ for $\mathcal{H}^s$-almost all $x$ in $F$. Then we can find an explicit constant $\underline{\theta}_{\ast}^s(F)\in [0,+\infty]$ such that 
\begin{equation}
\underline{\theta}_{\ast}^s(F) \mathcal{H}_s(\Pi_p(F)) \le \int_F \eta_s(F;x)d\mathcal{H}^s(x).
\end{equation}
We can say more: If $F$ is a $s$-Ahlfors-David regular set then $$0<\underline{\theta}_{\ast}^s(F)<+\infty.$$

\end{mydef3}
\begin{proof}

For $i\in G_n$ we write $a_i=\eta_s(F;i,n)$. Then, if $A_i=F\cap B_{p^{-n}}(i)$ we have   
   $$\eta_s(F;x,n)=\sum_{i\in G_n}a_i\chi_{A_i}(x), \ x\in F.$$
   We note that since $\eta_s(F;x,n)$ is bounded by an $\mathcal{H}^s$-integrable function we can use the reverse Fatou's Lemma, that is,
   $$\int_F \eta_s(F;x)d\mathcal{H}^s(x)=\int_F\limsup_n\eta_s(F;x,n)d\mathcal{H}^s(x)\ge \limsup_n\int_F\eta_s(F;x,n)d\mathcal{H}^s(x).$$
   Note that
   $$\int_F \eta_s(F;x,n)d\mathcal{H}^s(x)=\int_F\sum_{i\in G_n}a_i\chi_{A_i}(x)=\sum_{i\in G_n}a_i\mathcal{H}^s(A_i).$$
   Using Proposition \ref{prop.3.1} we have
   \begin{eqnarray*}
   \int_F \eta_s(F;x,n)d\mathcal{H}^s(x)=\sum_{i\in G_n}\mathcal{H}_s(F\cap B_{p^{-n}}(i)\cap \mathbb{Z})\dfrac{\mathcal{H}^s(A_i)}{p^{-ns}}\\ \ge \dfrac{\min_{i\in G_n}\mathcal{H}^s(F\cap B_{p^{-n}}(i))}{p^{-ns}}\cdot \sum_{i\in G_n}\mathcal{H}_s(F\cap B_{p^{-n}}(i)\cap \mathbb{Z}).
    \end{eqnarray*} 
   \textbf{Claim.} $\sum_{i\in G_n} \mathcal{H}_s(F\cap B_{p^{-n}}(i)\cap \mathbb{Z})\ge \mathcal{H}_s(F\cap \mathbb{Z})$.
   
   In fact, note that
    $\displaystyle \sum_{i\in G_n} \mathcal{H}(F\cap B_{p^{-n}}(i)\cap \mathbb{Z})=\sum_{i\in G_n} \limsup_{|I|\to \infty}\dfrac{|F\cap B_{p^{-n}}(i)\cap I|}{|I|^s}\ge \liminf_{|I|\to \infty} \sum_{i\in G_n}|F\cap B_{p^{-n}}(i)\cap I|=\limsup_{|I|\to \infty}\dfrac{|F\cap I|}{|I|^s}=\mathcal{H}_s(F\cap \mathbb{Z})$.
    Using the claim we have 
   $$\int \eta_s(F;x,n)d\mathcal{H}^s(x)\ge \dfrac{\min_{i\in G_n}\mathcal{H}^s(F\cap B_{p^{-n}}(i))}{p^{-ns}}\cdot \mathcal{H}_s(F\cap \mathbb{Z}).$$
    Taking the limsup we are done. The last statement follows from \ref{ADset1}.
    
   
\end{proof}

\begin{remark}
Let $F\subset \mathbb{Z}_p$ be a closed set and denote by $\mathcal{U}$ the set of all possible covers of $F$ by balls (each such cover is finite since the set is compact). Given $P\in \mathcal{U}$, we denote by $\text{diam}(P)$ the supremum of the diameters of the balls in $P$. Now, for $P\in \mathcal{U}$ define
$$\underline{\theta}^s_P(F)=\min_{B\in P} \dfrac{\mathcal{H}^s(F\cap B)}{\text{diam}(B)^s},$$ and define  

\begin{equation}\label{dM}\underline{\theta}^s_{\mathcal{U}}(F)=\lim_{\delta\to 0} \sup \{\underline{\theta}^s_P(F):\,P\in \mathcal{U},\, \text{diam}(P)<\delta\}.
\end{equation}
The proof of theorem C can easily be modified to prove that under the same hypothesis one has
\begin{equation*}
\int_F \eta_s(F;x)d\mathcal{H}^s(x)\ge \underline{\theta}_{\mathcal{U}}^s(F) \mathcal{H}_s(\Pi_p(F)).
\end{equation*}
\end{remark}

The following theorem is a consequence of theorems B and C together with proposition \ref{etas:teorianumeros}.

    \begin{mydef4}\label{cor1.3}\label{cor1.3}
    Let $F\subset \mathbb{Z}_p$ be a closed set such that $\mathcal{H}_s(\Pi_p(F))>0$. Suppose that $\mathcal{H}^s(F)<+\infty$. If either  
    \begin{enumerate}
    \item[(a)] there is a constant $C>0$ such that
\[\limsup_{|I|\to +\infty}\frac{|F\cap B\cap I|}{|B\cap I|^s} \le C,\]
for any ball $B\subset \mathbb{Z}_p$ or
    \item[(b)] There is an  $\mathcal{H}^s$-integrable function $g$ such that $\eta_s(F,x, n)\le g(x)$ for $\mathcal{H}^s$-almost all $x\in F$ and $F$ is $s$-Ahlfors-David regular.
    \end{enumerate} 
    then, for an unbounded sequence $\{n_j\}_j$ of natural numbers we have that $\Pi_p(F)\subset \mathbb{Z}$ is $s$-abundant in $\mathcal{A}^{p^{n_j}}_f$. 
\end{mydef4}

Now we relate the box-dimension of $F\subset \mathbb{Z}_p$ with the counting-dimension of $\Pi_p(E)\subset \mathbb{Z}$. 

\begin{mydef5}
   For any set $F\in \mathbb{Z}_p$ we have
    $$D(\Pi_p(F))\le BD(F).$$
In particular, for any set $E\subset \mathbb{Z}$ we have
    $$D(E)\le BD(E).$$
    Moreover, there is a closed set $F\subset \mathbb{Z}_5$ such that 
    $$D(\Pi_5(F))=BD(F)>\underline{BD}(F)\ge HD(F),$$
    where $\underline{BD}(F)$ is the lower Box dimension of $F$.
\end{mydef5}
\begin{proof}
Let $t< D(\Pi_p(F))$, then $\mathcal{H}_t(\Pi_p(F))=\infty$. We will show that there are infinite values of $n$ such that $N_{p^{-n}}(F)\geq p^{nt}$, which implies $BD(F)\geq t$ and hence $D(\Pi_p(F))\le BD(F)$. Suppose this not the case, then there is $n_0$ such that $N_{p^{-n}}(F)< p^{nt}$ for all $n>n_0$. Let $I$ be any interval such that $|I|>p^{n_0}$, then $p^{n-1}<|I|\leq p^{n}$ where $n$ is such that $N_{p^{-n}}(F)< p^{nt}$. Let $B_1,...,B_l$ be all the balls of radius $p^{-n}$ intersecting $F$, clearly $l\leq p^{nt}$. We get
\begin{align*}
\frac{|F\cap I|}{|I|^t} &\leq \sum_{j=1}^l \frac{|B_{j}\cap I|}{|I|^t}\\
&\leq p^{nt} \cdot \frac{(p^{-n}|I|+1)}{|I|^t}\leq  \frac{2p^{nt}}{(p^{n-1})^t}=2p^t.
\end{align*}
This last inequality contradicts the fact that $\mathcal{H}_t(\Pi_p(F))=\infty$ and we obtain the desired result.

     Next we proceed to the proof that there is a closed set $F\subset \mathbb{Z}_5$ such that $$D(\Pi_5(F))=BD(F)>HD(F).$$
     We can write $\mathbb{Z}_5=B_{5^{-1}}(0)\cup B_{5^{-1}}(1)\cup B_{5^{-1}}(2)\cup B_{5^{-1}}(3)\cup B_{5^{-1}}(4).$
     Consider $\mathcal{I}_1=\{0,2,4\}$ and $E_1=\cup_{i\in \mathcal{I}_1}B_{5^{-1}}(i)$. For any $i\in \mathcal{I}_1$ we have
     $B_{5^{-1}}(i)=\cup_{j=0}^{4} B_{5^{-2}}(i+j\cdot 5^1)$, then let $\mathcal{I}_2=\{i+j\cdot 5^1; i,j\in \mathcal{I}_1\}$ and $E_2=\cup_{i\in \mathcal{I}_2} B_{5^{-2}}(i).$ In the same way we write $\mathcal{I}_3=\{i+j\cdot 5^2;i\in \mathcal{I}_2, j\in \mathcal{I}_1\}$ and $E_3=\cup_{i\in \mathcal{I}_3}B_{5^{-3}}(i).$ Consider $\mathcal{I}_4=\{i+j\cdot 5^3; i\in \mathcal{I}_3, j\in \mathcal{I}_1\}$ and $E_4=\cup_{i\in \mathcal{I}_4}B_{5^{-4}}(i)$. Finally we write $\mathcal{I}_5=\{i+j\cdot 5^{4};i\in \mathcal{I}_4, j\in \mathcal{I}_1\}$ and $E_5=\cup_{i\in \mathcal{I}_5}B_{5^{-5}}(i)$. For $n\in \mathbb{N}$ let $k_n=5^n$. For $n\ge 1$ and $k_{2n-1}<k\le k_{2n}$ let $\mathcal{I}_k=\mathcal{I}_{k-1}$ and for $k_{2n}<k\le k_{2n+1}$ let $\mathcal{I}_k=\{i+j\cdot 5^{k-1};i\in \mathcal{I}_{k-1}, j\in \mathcal{I}_1\}.$ We consider $E_k=\cup_{i\in \mathcal{I}_k}B_{5^{-k}}(i)$. Note that $E_k\supset E_{k+1}$ and since $E_k$ is a compact set $F=\cap_{k\in \mathbb{N}} E_k$ is a non-empty compact set. Moreover, $\mathcal{I}_k\subset \mathcal{I}_{k+1}$. Note that $F\cap \mathbb{Z}=\cup_{k\in \mathbb{N}} \mathcal{I}_k$. We also note that $E_{k_n}\cap [0,5^{k_n}]=\mathcal{I}_{k_n}$. We observe that
     $$|F\cap [5^{k_{2n}},5^{k_{2n+1}}]|\ge 3^{k_{2n+1}-k_{2n}-2}$$
     and
     $$D(F\cap \mathbb{Z})\ge \limsup_{n\to \infty}\dfrac{(k_{2n+1}-k_{2n}-2)\log 3}{\log (5^{k_{2n+1}}-5^{k_{2n}})}\ge \dfrac{4}{5}\cdot \dfrac{\log 3}{\log 5}.$$
     On the other hand 
     $$\displaystyle N_{5^{-k_{2n+2}}}(F)=3^{k_1}\cdot 3^{k_3-k_2}\cdot ... \cdot 3^{k_{2n+1}-k_{2n}}=3^{\dfrac{k_{2n+2}+5}{6}}.$$
     In particular,
     $$\underline{BD}(F)\le \liminf_{n\to \infty} \dfrac{(k_{2n+2}+5)\log 3}{6 k_{2n+2}\log 5}<\dfrac{1}{6}\cdot \dfrac{\log 3}{\log 5}\le D(F\cap \mathbb{Z}).$$
     In particular, $HD(F)\le\underline{BD}(F)<D(F\cap \mathbb{Z})$
 \end{proof}
 
\begin{remark}
Notice that there was nothing especial about $p=5$ in the example in the previous theorem, we could have chosen any prime $p$.
\end{remark}

We finalize with Theorem F which gives us a necessary condition for a set $E\subset \mathbb{Z}$ belongs to $\mathcal{S}_p$ or $E\subset \mathbb{N}$ belong to $\mathcal{S}^{\ast}_p$.

\begin{mydef6}
Let $E\in \mathcal{S}_p$ or $E\in \mathcal{S}^{\ast}_p$ such that $BD(E)=s$. Then for any $\delta>0$, for $n\in \mathbb{N}$ large enough there are at most $k_n\le p^{n(s+\delta)}$ arithmetic progressions, $A_1,...,A_{k_n}$ of common difference $p^n$ such that $E\subset \cup_{i=1}^{k_n}A_i$, $\cup_{i=1}^{k_{n+1}}A_i\subset \cup_{i=1}^{k_n}A_i$ and
    $$E=\bigcap_{n\ge 1}\bigcup_{i=1}^{k_n}A_i.$$
\end{mydef6}

\begin{proof}
    We fix  $\Pi_p(\mathcal{F}_p(E))=\mathcal{F}_p(E)\cap \mathbb{Z}$, the case $\Pi_p(\mathcal{F}_p(E))=\mathcal{F}_p(E)\cap \mathbb{N}$ is quite similar. In the combinatorial proof of Theorem A, we have that the set $G_n:=G_n(F)=\{a\in \mathbb{Z}/p^n\mathbb{Z}; B_{p^{-n}}(a)\cap F\neq \emptyset \}$ play an important work. Using the definition of Box-counting dimension, we have that $|G_n(\mathcal{F}_p(E))|=O(p^{ns}).$ Moreover,
    $$\mathcal{F}_p(E)=\bigcap_{n\ge 1}\bigcup_{a\in G_n}B_{p^{-n}}(a).$$
    This implies
    $$E=\Pi_p(\mathcal{F}_p(E))=\bigcap_{n\ge 1}\bigcup_{a\in G_n}B_{p^{-n}}(a)\cap \mathbb{Z}.$$
    Any set $A_j=B_{p^{-n}}(j)\cap \mathbb{Z}$ is an arithmetic progression and by definition of Box-counting dimension for any $\delta>0$ for $n$ large enough we have $|G_n|=k_n\le p^{n(s+\delta)}$.
\end{proof}
\section{Cantor sets in $\mathbb{Z}$}

    In this section we present a class of subsets which provide non trivial examples for the development of this theory. First we start with the classical ternary Cantor in $\mathbb{N}$ which is given in the same way as the ternary Cantor set in $[0,1]$, that is, the set $K\subset [0,1]$  such that any point in $K$ has digits $0$ and $2$ in its expansion in base 3. We can define the ternary Cantor set in $\mathbb{N}$ as
    $$E=\{x=x_0+x_1\cdot 3^1+x_2\cdot 3^2+...+x_n\cdot 3^n; n\in \mathbb{N}, \ x_i\in \{0,2\}\}.$$
    In particular,
    $$\mathcal{F}_3(E)=\left\{\sum_{i\ge 0}x_i3^i; x_i\in \{0,2\}\right\}$$
    and  by Lemma \ref{inv1} we have $\Pi_3(\mathcal{F}_3(E))=\mathcal{F}_3(E)\cap \mathbb{N}=E.$
    Note that the $p$-adic ball $B_{3^{-n}}(x)\cap \mathcal{F}_3(E)\neq \emptyset$ iff $x\in E$. Using this, we can compute the number of $z=\sum_{i=0}^N a_i\cdot 3^i\in B_{3^{-n}}(x)\cap \mathcal{F}_3(E)\cap (3^N,3^{N+1}]$. In fact, $z\equiv x \ (\mbox{mod.} \ 3^n)$ implies that we can choose $(a_n,a_{n+1},...,a_N)\in \{0,2\}^{N-n+1}$. Then
    $$|B_{3^{-n}}(x)\cap \mathcal{F}_3(E)\cap (3^N,3^{N+1}]|\le 2^{N-n+1}.$$
    If $I=(M,N]$ we take $k$ and $j$ such that $3^k<M\le 3^{k+1}$ and $3^j<N\le 3^{j+1}$ and then
    $I=(M,3^k]\cup (3^{k},3^{k+1}]\cup...\cup (3^{j-1},3^j]\cup (3^j,N].$ Therefore
    $$|B_{3^{-n}}(x)\cap \mathcal{F}_3(E)\cap I|\le \sum_{i=0}^{j+1-k} |B_{3^{-n}}(x)\cap \mathcal{F}_3(E)\cap (3^{k-1+i},3^{k+i}]|\le \sum_{i=0}^{j+1-k} 2^{k+i-n}.$$
    In particular, $|B_{3^{-n}}(x)\cap \mathcal{F}_3(E)\cap I|\le 2^{j-n+1}.$ It follows that, for $s=\dfrac{\log 2}{\log 3}$ we have
    $$\eta_s(\mathcal{F}_3(E);x,n)=\limsup_{|I|\to \infty}\dfrac{|B_{3^{-n}}(x)\cap \mathcal{F}_3(E)\cap I|}{|B_{3^{-n}}\cap I|}\le \dfrac{2^{j-n+1}}{3^{(j-1-n)s}}\le c, \ \forall \ n\in \mathbb{N}$$
     and for every  $x\in \mathbb{Z}_3,$ where $c$ is a positive constant. Note that if $K\subset [0,1]$ is the ternary Cantor set and $g:K\rightarrow \mathcal{F}_3(E)$ given by $g\left(\sum_{i\ge 0}\dfrac{x_i}{3^i}\right)=\sum_{i\ge 0}x_i\cdot 3^i$ then $g$ is a bi-Lipschitz map and then  $\mathcal{H}^s(\mathcal{F}_3(E))<+\infty$, the map $\eta_s(\mathcal{F}_3(E);\cdot, n)$ is bounded by an $\mathcal{H}^s$-integrable function. Moreover, $\mathcal{F}_3(E)$ is $s$-Ahlfors-David regular set. Thus, $\mathcal{F}_3(E)$ satisfies both Theorems B and C. Then, since $E=\Pi_p(\mathcal{F}_p(E))$ by Theorem D there is a sequence of arithmetic progressions $\{A_k\}$ each with common difference $3^{n_k}$, for $n_k\in \mathbb{N}$ and $|A_{k+1}|>|A_k|$ such that 
     $$|E\cap A_k|\ge c|A_k|^s, s=\dfrac{\log 2}{\log 3}.$$
     Moreover, by Theorem F, for each $n$ there are $k_n=O(p^{ns})$ arithmetic progressions, $A^{(n)}_1,...,A^{(n)}_{k_n}$ of common difference $3^{n}$ such that $E\subset \bigcup_{i=1}^{k_n}A^{(n)}_i$ and
     $$E\subset \bigcap_{n\ge 1}\bigcup_{i=1}^{k_n}A^{(n)}_i.$$

      This example can be generalized for any $p$ prime. Given a binary matrix $A=(a_{ij})_{0\le i,j\le p-1}$ we write
        $$E_A=\{x=\sum_{i=0}^n x_i\cdot p^i; n\in \mathbb{N}, a_{x_i,x_{i+1}}=1, \ \forall \ i\ge 1\}.$$
    In the same way 
        $$\mathcal{F}_p(E_A)=\{x=\sum_{i\ge 0} x_i\cdot p^i;a_{x_i,x_{i+1}}=1, \ \forall \ i\ge 1\}.$$
        In \cite{YG} Lima and Moreira showed that $D(E_A)=\dfrac{\log \lambda^{+}(A)}{\log p}:=s$ and $0<\mathcal{H}_s(E_A)<+\infty$, where $\lambda^{+}(A)$ is the Perron eigenvalue of $A$.
        In the similar way as above we can prove that for $s=\dfrac{\log \lambda^{+}(A)}{\log p}$,  there exists $C>0$ such that $\eta_s(\mathcal{F}_p(E_A))<C$ and $\mathcal{H}^s(\mathcal{F}_p(E_A))<C$. In particular, $\eta_s(\mathcal{F}_p(E_A))$ is an $\mathcal{H}^s$-integrable function. Moreover, $\mathcal{F}_p(E_A)$ is the image of a bi-Lipschitz map of a regular affine Cantor set on $\mathbb{R}$ with Hausdorff dimension $\dfrac{\log \lambda^+(A)}{\log p}.$ In particular for this $s$ we are in the hypoteses of our results in Theorems B and C, that is, every set $\mathcal{F}_p(E_A)$ is $s$-Ahlfors-David regular set and then  $\underline{\theta}^s_{\ast}(\mathcal{F}_p(E_A))\in (0,+\infty).$

This class of sets is such that $\Pi_p(\mathcal{F}_p(E))=\mathcal{F}_p(E)\cap \mathbb{N}=E$. Moreover, Cantor sets satisfy 
\begin{equation}\label{HD=BD}
HD(\mathcal{F}_p(E_A))=BD(E_A))=D(E_A).
\end{equation}
     All the results that we give for $E$, the ternary cantor set in $\mathbb{Z}$ we have for this class $E_A$ of Cantor sets in $\mathbb{Z}$. 

\subsection{Acknowledgements}
 The authors also thank to Paulo Ribenboim for several discussions about densities and Carlos G. Moreira and Y. Lima for his suggestions to improve the work.
The first author also thanks to CNPq for the grant - Alagoas Din\^amica 409198/2021-8 and FAPEAL grant E60030.0000002330/2022. The second author also would like to thank CNPq and FAPERJ for financial support.

\bibliographystyle{amsplain}

\end{document}